\documentclass[review,3p,times,sort&compress]{elsarticle}

\usepackage{amsmath,amssymb,amsfonts,amsthm,bm}
\usepackage{graphicx}
\usepackage{subcaption}
\usepackage{multirow}
\usepackage{booktabs}
\usepackage{mathrsfs,float}
\usepackage{xcolor}
\usepackage[colorlinks,citecolor=blue,linkcolor=blue,urlcolor=blue]{hyperref}
\usepackage{algorithm}
\usepackage{algorithmic}
\journal{Applied Mathematics and Computation}

\newtheorem{theorem}{Theorem}[section]
\newtheorem{remark}[theorem]{Remark}

\newtheorem{proposition}[theorem]{Proposition}

\newtheorem{lemma}[theorem]{Lemma}

\begin{document}

\begin{frontmatter}

\title{A Localized Fourier Extension Method for Piecewise-Smooth Inverse Source Reconstruction}

\author[2]{Zhihong Dou}
\author[1]{Zhenyu Zhao}

\address[1]{School of Mathematics and Statistics, Shandong University of Technology, Zibo 255049, China}
\address[2]{Department of Basic Courses, Langfang Campus, Hebei University of Technology, Langfang 065099, China}

\begin{abstract}
We reconstruct a piecewise-smooth source in a Poisson equation on a
semi-infinite strip from noisy solution values measured along an interior
line. Applying the one-dimensional Dirichlet Laplacian to the observation
reduces the inverse problem to regularized second-order differentiation
followed by a boundedly invertible correction. The differentiated trace and
the source differ by an analytic smoothing term and therefore have the same
interior singular support and jump data.

This structure motivates a localized Fourier extension method that combines
two staggered detection partitions, GTSVD-regularized local derivative
coefficients, mollified conjugate Fourier sums, and structure-aligned
numerical differentiation. The source is then recovered by an exponentially
decaying spectral correction. For an exact partition, the reconstruction
inherits the piecewise differentiation rate
$\mathcal O(\delta^{(\bar s-2)/\bar s})$; local peak and partition
perturbation estimates describe the additional effect of breakpoint errors.
Comparisons with full-grid total-variation regularization and truncated
Fourier inversion show competitive high-noise performance and a pronounced
low-noise advantage of the localized method. Repeated Gaussian-noise tests
confirm robustness in the moderate- and low-noise regimes.
\end{abstract}


\end{frontmatter}

\section{Introduction}\label{sec:introduction}

Inverse source problems recover an inaccessible forcing term from indirect
measurements of a state governed by a partial differential equation. They
arise in heat conduction, pollutant detection, crack identification,
source localization, and thermo- and photo-acoustic imaging
\cite{CannonDuchateau1998,AlbaneseMonk2006,LiuUhlmann2015,Isakov2006}.
Because the forward map smooths the source, regularization is required
\cite{EnglHankeNeubauer1996}.

We consider the Poisson equation on a semi-infinite strip,
\begin{equation}\label{eq:forward_problem}
\left\{
\begin{aligned}
-u_{xx}(x,y)-u_{yy}(x,y)&=f(x),
&&0<x<\pi,\quad 0<y<\infty,\\
u(0,y)=u(\pi,y)&=0,
&&0\leq y<\infty,\\
u(x,0)&=0,
&&0\leq x\leq\pi,\\
u(x,y)&\ \text{is bounded as }y\to\infty .
\end{aligned}
\right.
\end{equation}
The source is reconstructed from the interior trace
\begin{equation}\label{eq:observation}
g(x)=u(x,y_0),\qquad 0<x<\pi,
\end{equation}
where $y_0>0$, from noisy data satisfying
\begin{equation}\label{eq:noise}
\|g^\delta-g\|_{L^2(0,\pi)}\leq\delta .
\end{equation}
For this model, spectral truncation, modified regularization,
quasi-reversibility, Hilbert-scale Tikhonov regularization, super-order
regularization, and Fourier extension have been studied
\cite{Yang2011,YangFu2012,ZhaoEtAl2014,LiEtAl2019,WenHuangLiu2021,Zhao2022}.
Most of these methods are designed for globally smooth sources.

Here $f$ is piecewise smooth and its breakpoints are unknown. Global spectral
representations then suffer from low global regularity and Gibbs oscillations,
whereas global smoothness penalties tend to blur interfaces. Total-variation
(TV) regularization is a standard structure-preserving alternative that does
not require explicit breakpoint detection
\cite{RudinOsherFatemi1992,AcarVogel1994,BenningBurger2018}; it is included
below as a full-grid benchmark. Discontinuous source identification has also
been considered in other settings
\cite{HettlichRundell2001,WangJiaCheng2002,WanWangYamamoto2006}, but the
unknown interfaces in \eqref{eq:forward_problem}--\eqref{eq:observation}
have not been incorporated into a localized reconstruction procedure.

The key structural observation follows from the Dirichlet operator
\[
L=-\frac{d^2}{dx^2},\qquad D(L)=H^2(0,\pi)\cap H_0^1(0,\pi).
\]
The observation operator satisfies
\[
g=L^{-1}\bigl(I-e^{-y_0\sqrt L}\bigr)f,
\]
and hence
\begin{equation}\label{eq:operator_decomposition}
q:=-g''=\bigl(I-e^{-y_0\sqrt L}\bigr)f.
\end{equation}
Thus the order-two instability is confined to numerical differentiation,
while the remaining operator is boundedly invertible. Moreover,
$q-f=-e^{-y_0\sqrt L}f$ is analytic in the interior, so $q$ and $f$ have the
same interior singular support and jump data. This reduction is independent
of the regularization used for the differentiation step.

To exploit it, we combine Fourier extension approximation
\cite{Huybrechs2010,AdcockEtAl2014,AdcockHuybrechs2019} with localized and
multi-interval constructions
\cite{ZhaoWang2026,ZhaoWangLiu2025}. Regularized derivative coefficients are
computed by GTSVD, and a mollified conjugate Fourier sum
\cite{GelbTadmor1999,GelbTadmor2000,CochranGelbWang2013} detects the
breakpoints. Two staggered fixed-length partitions prevent a breakpoint from
being hidden at every local endpoint. The measured trace is then
differentiated independently on the detected smooth components, and a smooth
spectral correction recovers the source. This extends the use of regularized
Fourier coefficients for edge detection in \cite{ZhaoYuJiaDou2025} from
numerical differentiation to a complete inverse-source reconstruction.

The main contributions are:
\begin{enumerate}
\item an operator reduction and structure-preservation theorem showing that
$-g''$ and $f$ have identical interior singular support and jumps;
\item a unified localized algorithm for automatic breakpoint detection,
structure-aligned differentiation, and bounded source correction, together
with local peak and interface-perturbation estimates;
\item numerical comparisons with truncated Fourier and full-grid TV methods,
including repeated Gaussian-noise tests that identify the accuracy and
stability regimes of the complete pipeline.
\end{enumerate}

Section~\ref{sec:direct} establishes the operator and structural results.
Section~\ref{sec:method} develops the reconstruction and its error estimates.
Section~\ref{sec:numerics} presents the comparisons and robustness tests, and
Section~\ref{sec:conclusion} concludes the paper.

\section{Operator formulation and structural properties}
\label{sec:direct}

Let
\[
\mathcal{H}=L^2(0,\pi)
\]
with the usual inner product $(\cdot,\cdot)$ and norm $\|\cdot\|$. We
introduce the Dirichlet operator
\begin{equation}\label{eq:def_L}
L=-\frac{d^2}{dx^2},
\qquad
D(L)=H^2(0,\pi)\cap H_0^1(0,\pi).
\end{equation}
Its normalized eigenfunctions and eigenvalues are
\[
\phi_n(x)=\sqrt{\frac{2}{\pi}}\sin(nx),
\qquad
L\phi_n=n^2\phi_n,
\qquad n=1,2,\ldots .
\]
For $f\in\mathcal{H}$, we write
\[
f=\sum_{n=1}^{\infty}f_n\phi_n,
\qquad
f_n=(f,\phi_n).
\]
By the standard separation-of-variables argument \cite{Yang2011}, the unique bounded
solution of the direct problem is given by
\begin{equation}\label{eq:direct_solution}
u(x,y)
=
\sum_{n=1}^{\infty}
\frac{1-e^{-ny}}{n^2}f_n\phi_n(x).
\end{equation}
Equivalently,
\begin{equation}\label{eq:forward_operator}
u(\cdot,y)=A_yf,
\qquad
A_y=L^{-1}\bigl(I-e^{-y\sqrt L}\bigr).
\end{equation}

For a fixed observation height $y_0>0$, the measured data
\[
g(x)=u(x,y_0)
\]
satisfy
\begin{equation}\label{eq:observation_operator}
g=A_{y_0}f.
\end{equation}
Since
\[
A_{y_0}\phi_n
=
\frac{1-e^{-ny_0}}{n^2}\phi_n,
\]
the operator $A_{y_0}$ is compact, self-adjoint, positive, and injective,
with singular values
\begin{equation}\label{eq:singular_values}
\mu_n(y_0)
=
\frac{1-e^{-ny_0}}{n^2},
\qquad n=1,2,\ldots .
\end{equation}
For every fixed $y_0>0$,
\[
\mu_n(y_0)\sim n^{-2}
\qquad\text{as }n\to\infty.
\]
Therefore, the recovery of $f$ from $g$ is a mildly ill-posed inverse
problem of order two.

\subsection{Decomposition of the inverse problem}
\label{subsec:decomposition}

The spectral representation \eqref{eq:singular_values} suggests that the
factor $n^2$ is responsible for the instability of the inverse problem.
To make this observation explicit, we apply $L$ to the exact observation
and define
\begin{equation}\label{eq:def_q}
q:=Lg=-g''.
\end{equation}
It follows from \eqref{eq:forward_operator} that
\begin{equation}\label{eq:By_equation}
q=B_{y_0}f,
\qquad
B_{y_0}:=I-e^{-y_0\sqrt L}.
\end{equation}

The operator $B_{y_0}$ contains no frequency-dependent loss of
stability. Indeed,
\[
B_{y_0}\phi_n
=
(1-e^{-ny_0})\phi_n,
\]
and its eigenvalues are uniformly separated from zero.

\begin{proposition}\label{prop:By}
For every fixed $y_0>0$, the operator
\[
B_{y_0}=I-e^{-y_0\sqrt L}
\]
is bounded, self-adjoint, positive, and boundedly invertible on
$L^2(0,\pi)$. Moreover,
\begin{equation}\label{eq:By_bounds}
(1-e^{-y_0})\|v\|^2
\leq
(B_{y_0}v,v)
\leq
\|v\|^2,
\qquad v\in L^2(0,\pi),
\end{equation}
and
\begin{equation}\label{eq:By_inverse_norm}
\|B_{y_0}^{-1}\|_{\mathcal{L}(L^2,L^2)}
=
\frac{1}{1-e^{-y_0}}.
\end{equation}
\end{proposition}

\begin{proof}
For
\[
v=\sum_{n=1}^{\infty}v_n\phi_n,
\]
one has
\[
(B_{y_0}v,v)
=
\sum_{n=1}^{\infty}
(1-e^{-ny_0})|v_n|^2.
\]
Since
\[
1-e^{-y_0}
\leq
1-e^{-ny_0}
<1,
\qquad n\geq1,
\]
inequality \eqref{eq:By_bounds} follows immediately. In particular,
$B_{y_0}$ is boundedly invertible, and
\[
B_{y_0}^{-1}v
=
\sum_{n=1}^{\infty}
\frac{v_n}{1-e^{-ny_0}}\phi_n.
\]
The largest inverse multiplier is attained at $n=1$, which gives
\eqref{eq:By_inverse_norm}.
\end{proof}

Consequently, the inverse source problem admits the decomposition
\begin{equation}\label{eq:two_step_inverse}
g^\delta
\quad
\xrightarrow{\text{regularized second differentiation}}
\quad
q^\delta
\quad
\xrightarrow{\,B_{y_0}^{-1}\,}
\quad
f^\delta .
\end{equation}
The first step is ill-posed and requires regularization, whereas the
second step is well posed for every fixed $y_0>0$. Thus, the entire
frequency-dependent instability of the inverse source problem is
concentrated in the recovery of the second derivative of the measured
data.

This decomposition also gives a direct error-transfer estimate. If
$\widetilde q$ is an approximation to $q$ and
\[
\widetilde f=B_{y_0}^{-1}\widetilde q,
\]
then
\begin{equation}\label{eq:error_transfer}
\|f-\widetilde f\|
\leq
\frac{1}{1-e^{-y_0}}
\|q-\widetilde q\|.
\end{equation}
Therefore, any convergence estimate for the regularized approximation of
$-g''$ immediately yields a corresponding estimate for the recovered
source.

\begin{remark}\label{rem:height}
Although $B_{y_0}^{-1}$ is bounded for every fixed $y_0>0$, its norm
satisfies
\[
\|B_{y_0}^{-1}\|
=
\frac{1}{1-e^{-y_0}}
\sim \frac{1}{y_0}
\qquad\text{as }y_0\to0.
\]
Hence, measurements taken very close to the boundary $y=0$ may lead to a
larger amplification of the low-frequency errors. This does not change
the order-two ill-posedness of the original problem, but it influences
the conditioning of the well-posed correction step.
\end{remark}

For later use, it is also convenient to write
\[
\frac{1}{1-e^{-ny_0}}
=
1+\frac{1}{e^{ny_0}-1}.
\]
Defining
\[
R_{y_0}:=(e^{y_0\sqrt L}-I)^{-1},
\]
we obtain
\begin{equation}\label{eq:f_q_decomposition}
f=q+R_{y_0}q.
\end{equation}
The eigenvalues of $R_{y_0}$ are
\[
\frac{1}{e^{ny_0}-1},
\]
which decay exponentially as $n\to\infty$. Thus, $R_{y_0}q$ is a smooth
correction to $q$. This observation is particularly useful for analyzing
the local singularities of the source.

\subsection{Preservation of piecewise-smooth structures}
\label{subsec:structure}

We now investigate the relation between the piecewise-smooth structures
of $f$ and
\[
q=B_{y_0}f.
\]
Let
\[
P_{y_0}:=e^{-y_0\sqrt L}.
\]
Then
\begin{equation}\label{eq:q_f_relation}
q=f-P_{y_0}f.
\end{equation}
The operator $P_{y_0}$ has an exponential smoothing effect in the
$x$-variable.

\begin{lemma}\label{lem:Poisson_smoothing}
For every $f\in L^2(0,\pi)$ and $y_0>0$, the function
\[
P_{y_0}f
=
\sum_{n=1}^{\infty}e^{-ny_0}f_n\phi_n
\]
is real analytic in the interior of $(0,\pi)$.
\end{lemma}

\begin{proof}
Let $0<\rho<y_0$. For a complex variable $z$ satisfying
$|\operatorname{Im}z|\leq\rho$, one has
\[
|\sin(nz)|\leq C_\rho e^{n\rho}.
\]
Hence,
\[
|e^{-ny_0}f_n\phi_n(z)|
\leq
C_\rho |f_n|e^{-n(y_0-\rho)}.
\]
The series on the right-hand side is summable by the
Cauchy--Schwarz inequality. Therefore, the Fourier series defining
$P_{y_0}f$ converges uniformly on compact subsets of
\[
\{z\in\mathbb{C}:|\operatorname{Im}z|<y_0\},
\]
and its sum is holomorphic there.
\end{proof}

Suppose that $f$ is piecewise smooth with a finite set of unknown
breakpoints
\begin{equation}\label{eq:source_partition}
0<\xi_1<\xi_2<\cdots<\xi_J<\pi.
\end{equation}
For a function $v$ with one-sided limits at $\xi_j$, we denote
\[
[v]_{\xi_j}
=
v(\xi_j^+)-v(\xi_j^-).
\]

\begin{theorem}\label{thm:structure_preservation}
Let $f\in L^2(0,\pi)$ be piecewise $C^m$ on the partition determined by
\eqref{eq:source_partition}, and let
\[
q=(I-P_{y_0})f
\]
for some fixed $y_0>0$. Then $f$ and $q$ have the same interior singular
support:
\begin{equation}\label{eq:singular_support}
\operatorname{sing\,supp}(q)
=
\operatorname{sing\,supp}(f).
\end{equation}
Moreover, at each breakpoint $\xi_j$,
\begin{equation}\label{eq:jump_preservation}
[q^{(k)}]_{\xi_j}
=
[f^{(k)}]_{\xi_j},
\qquad k=0,1,\ldots,m,
\end{equation}
whenever the corresponding one-sided derivatives exist.
\end{theorem}

\begin{proof}
By Lemma~\ref{lem:Poisson_smoothing}, $P_{y_0}f$ is analytic in
$(0,\pi)$. Therefore,
\[
[(P_{y_0}f)^{(k)}]_{\xi_j}=0
\]
for every $k\geq0$. It follows from \eqref{eq:q_f_relation} that
\[
[q^{(k)}]_{\xi_j}
=
[f^{(k)}]_{\xi_j}
-
[(P_{y_0}f)^{(k)}]_{\xi_j}
=
[f^{(k)}]_{\xi_j}.
\]
This proves \eqref{eq:jump_preservation}.

If $f$ is smooth in a neighborhood of an interior point, then
$q=f-P_{y_0}f$ is also smooth there. Conversely, if $q$ were smooth in a
neighborhood of a singular point of $f$, then
\[
f=q+P_{y_0}f
\]
would also be smooth in that neighborhood, since $P_{y_0}f$ is analytic.
This contradiction proves \eqref{eq:singular_support}.
\end{proof}

Theorem~\ref{thm:structure_preservation} shows that $q=-g''$ retains
the location, order, and amplitude of every interior singularity of $f$.
This permits the breakpoints to be detected from a regularized local
approximation of $q$, after which differentiation is performed separately on
the smooth components and the source is recovered from the well-posed
equation $B_{y_0}f=q$.

\begin{remark}[Differential preconditioning]\label{rem:differential_preconditioning}
The reduction is not specific to Fourier extension. If an inverse source map
$g=\mathcal Af$ admits an operator $D$ such that
$D\mathcal A=I-\mathcal S$, where $\mathcal S$ smooths across the interfaces
and $I-\mathcal S$ is boundedly invertible, then $Dg$ and $f$ share their
interior jump data. Whether such a factorization exists depends on the PDE,
observation geometry, and source parameterization.
\end{remark}

\section{Localized Fourier extension method for source reconstruction}
\label{sec:method}

Section~\ref{sec:direct} reduces the inverse problem to regularized
recovery of $q=-g''$ followed by $B_{y_0}f=q$. The proposed method uses the
same localized Fourier extension framework for two tasks: detecting the
breakpoints by regularized conjugate sums on staggered partitions, and
computing the second derivative independently on the resulting smooth
components. Fixed normalized grids allow the local matrices and their GTSVD
factorizations to be reused. The source is finally obtained from a smooth
spectral correction.

\subsection{Discrete observations and the localized Fourier extension framework}
\label{subsec:local_LFE}

Let
\begin{equation}\label{eq:sampling_nodes}
x_i=ih,
\qquad
h=\frac{\pi}{M},
\qquad
i=0,1,\ldots,M,
\end{equation}
be a uniform grid on $[0,\pi]$. The measured data are assumed to have the
form
\begin{equation}\label{eq:discrete_noisy_data}
g_i^\delta=g(x_i)+\eta_i,
\qquad
i=0,1,\ldots,M,
\end{equation}
where $\eta_i$ denotes the measurement noise. We write
\[
\mathbf g^\delta
=
(g_0^\delta,g_1^\delta,\ldots,g_M^\delta)^{\mathrm T}.
\]

Consider a local interval
\[
I=[a,b]\subset[0,\pi].
\]
For an extension parameter $T>1$, it is mapped onto
\[
\Lambda_T=\left[0,\frac{2\pi}{T}\right]
\]
by
\begin{equation}\label{eq:local_mapping}
t=t_I(x)
=
\frac{2\pi(x-a)}{T(b-a)}.
\end{equation}
Let
\[
\varphi_\ell(t)=e^{\mathrm i\ell t},
\qquad |\ell|\leq n,
\]
and approximate $g$ on $I$ by the local Fourier extension
\begin{equation}\label{eq:local_FE}
p_I(x)
=
\sum_{\ell=-n}^{n}
c_{\ell,I}\,
\varphi_\ell\bigl(t_I(x)\bigr).
\end{equation}

Suppose that $m$ local sampling points are used. The corresponding discrete
Fourier extension system is
\begin{equation}\label{eq:local_FE_system}
F\mathbf c_I=\mathbf g_I^\delta,
\end{equation}
where
\[
F_{r,\ell}=\varphi_\ell(t_r),
\qquad
\mathbf c_I=(c_{-n,I},\ldots,c_{n,I})^{\mathrm T}.
\]
Intervals belonging to the same computational stage have the same length,
reference mapping, and normalized sampling pattern. Their local matrices are
therefore identical. The corresponding matrix decomposition can be
precomputed and reused throughout that stage
\cite{ZhaoWang2026,ZhaoWangLiu2025}.

For numerical differentiation, we use the weighted GTSVD regularization
introduced in \cite{ZhaoWangLiu2025}. Let
\[
W_{\rm rec}=\operatorname{diag}\{e^{|\ell|}\}_{\ell=-n}^{n}
\]
and suppose that
\begin{equation}\label{eq:G_svd}
FW_{\rm rec}^{-1}
=
U\Sigma V^*,
\qquad
\Sigma=\operatorname{diag}(\sigma_1,\ldots,\sigma_{2n+1}),
\end{equation}
with
\[
\sigma_1\geq\sigma_2\geq\cdots\geq\sigma_{2n+1}\geq0.
\]
For a truncation index $\nu$, the local coefficient vector is defined by
\begin{equation}\label{eq:GTSVD_coefficients}
\mathbf c_{I,\nu}^{\delta}
=
W_{\rm rec}^{-1}
\sum_{r=1}^{\nu}
\frac{\langle\mathbf g_I^\delta,\mathbf u_r\rangle}
{\sigma_r}\mathbf v_r.
\end{equation}
The index $\nu$ is selected according to a local discrepancy principle,
\begin{equation}\label{eq:local_discrepancy}
\left\|
F\mathbf c_{I,\nu}^{\delta}-\mathbf g_I^\delta
\right\|_2
\leq
\tau_{\rm rec}\delta_I
<
\left\|
F\mathbf c_{I,\nu-1}^{\delta}-\mathbf g_I^\delta
\right\|_2,
\qquad \tau_{\rm rec}>1,
\end{equation}
where $\delta_I$ is an estimate of the discrete noise norm on the current
local interval. If the target residual cannot be attained using the reliable
singular components, all components satisfying
$\sigma_r\geq\epsilon_{\rm sv}\sigma_1$ are retained.

The derivatives of the local Fourier extension are evaluated analytically.
Since
\[
\frac{d^2}{dx^2}
\varphi_\ell\bigl(t_I(x)\bigr)
=
-
\left(
\frac{2\pi\ell}{T(b-a)}
\right)^2
\varphi_\ell\bigl(t_I(x)\bigr),
\]
the local approximation of $q=-g''$ is
\begin{equation}\label{eq:local_q_reconstruction}
q_I^\delta(x)
=
\sum_{\ell=-n}^{n}
\left(
\frac{2\pi\ell}{T(b-a)}
\right)^2
c_{\ell,I}^{\delta}
\varphi_\ell\bigl(t_I(x)\bigr).
\end{equation}
Formula~\eqref{eq:local_q_reconstruction} also explains why excessive
subdivision should be avoided. Although localization reduces the effective
oscillation on each interval, the second derivative introduces the scaling
factor $(b-a)^{-2}$. Additional subdivision is therefore introduced only
when the current local Fourier space cannot represent the data to the
prescribed tolerance.

\begin{remark}[Stage-dependent extension parameters]
\label{rem:stage_dependent_T}
Detection and differentiation use the same local representation but have
different objectives. We use $T_{\rm det}=2$ to retain detector bandwidth
and $T_{\rm rec}=6$ to reduce endpoint effects in smooth-component
differentiation. Both values are fixed for all tests, and each associated
matrix factorization is reused.
\end{remark}

\subsection{Localized breakpoint detection on staggered partitions}
\label{subsec:breakpoint_detection}

According to Theorem~\ref{thm:structure_preservation}, $q=-g''$ and $f$
have the same interior singular support. The breakpoints can therefore be
detected directly from the noisy observation $g^\delta$ without first
constructing a global approximation of $q$.

Let
\[
H=\frac{\pi}{K}
\]
be a prescribed localization length. We introduce the primary partition
\begin{equation}\label{eq:primary_detection_partition}
\mathcal P_0
=
\{I_k\}_{k=1}^{K},
\qquad
I_k=[(k-1)H,kH],
\end{equation}
and the staggered partition
\begin{equation}\label{eq:staggered_detection_partition}
\mathcal P_{1/2}
=
\{\widetilde I_k\}_{k=1}^{K-1},
\qquad
\widetilde I_k
=
\left[
\left(k-\frac12\right)H,
\left(k+\frac12\right)H
\right].
\end{equation}
The detection stage therefore uses only $2K-1$ fixed-length local intervals.
The second family is not introduced as an independent moving-window
procedure. It is a half-cell translation of the primary local partition and
has the sole purpose of resolving the ambiguity at local interfaces. Every
interior interface of $\mathcal P_0$ lies in the interior of an interval of
$\mathcal P_{1/2}$, and conversely. Hence, an interior breakpoint lies away
from the local endpoints in at least one of the two representations.

\subsubsection{Coarse localization by local conjugate sums}

On each interval $J=[a_J,b_J]\in\mathcal P_0\cup\mathcal P_{1/2}$, we
select the same number $m_c$ of equally spaced data nodes and use a fixed
low-order Fourier space with $n_c$ modes. For the detector, it is convenient
to use the real representation
\begin{equation}\label{eq:detection_local_FE}
p_J^\delta(x)
=
a_{0,J}
+
\sum_{\ell=1}^{n_c}
\left[
a_{\ell,J}\cos\bigl(\ell t_J(x)\bigr)
+
b_{\ell,J}\sin\bigl(\ell t_J(x)\bigr)
\right],
\end{equation}
where $t_J$ is given by \eqref{eq:local_mapping} with
$T=T_{\rm det}$.

The detector requires stable Fourier coefficients of a second derivative.
We therefore use a Sobolev-weighted GTSVD. Let
\[
A_c
=
\begin{bmatrix}
\cos(t_r) & \cdots & \cos(n_ct_r) &
\sin(t_r) & \cdots & \sin(n_ct_r)
\end{bmatrix}_{r=1}^{m_c},
\]
\[
P_c
=
I-\frac{1}{m_c}\mathbf 1\mathbf 1^{\mathrm T},
\qquad
D_p
=
\operatorname{diag}
\bigl(1^p,\ldots,n_c^p,1^p,\ldots,n_c^p\bigr),
\]
where $p=2+\beta$ with $0<\beta<1/2$. We precompute
\begin{equation}\label{eq:detection_GTSVD}
P_cA_cD_p^{-1}
=
U_c\Sigma_cV_c^{\mathrm T}.
\end{equation}
For the centered local data
$\mathbf y_{J,c}^\delta=P_c\mathbf g_{J,c}^\delta$, the truncated
coefficient vector is
\begin{equation}\label{eq:detection_coefficients}
\mathbf z_{J,\nu}^{\delta}
=
D_p^{-1}
\sum_{r=1}^{\nu}
\frac{\langle\mathbf y_{J,c}^\delta,\mathbf u_{c,r}\rangle}
{\sigma_{c,r}}\mathbf v_{c,r},
\end{equation}
where
\[
\mathbf z_{J,\nu}^{\delta}
=
(a_{1,J},\ldots,a_{n_c,J},b_{1,J},\ldots,b_{n_c,J})^{\mathrm T}.
\]
The constant coefficient is recovered from
\begin{equation}\label{eq:detection_constant_coefficient}
a_{0,J}
=
\frac{1}{m_c}\mathbf 1^{\mathrm T}
\left(
\mathbf g_{J,c}^\delta-A_c\mathbf z_{J,\nu}^{\delta}
\right).
\end{equation}
The smallest reliable truncation index satisfying
\begin{equation}\label{eq:detection_discrepancy}
\left\|
a_{0,J}\mathbf 1+A_c\mathbf z_{J,\nu}^{\delta}
-
\mathbf g_{J,c}^\delta
\right\|_2
\leq
\tau_{\rm det}\delta_{J,c}
\end{equation}
is selected. If the discrepancy target cannot be attained, the largest
reliable index determined by
$\sigma_{c,r}\geq\epsilon_{\rm sv}\sigma_{c,1}$ is used. Since every coarse
interval has the same length, number of modes, and normalized sampling
pattern, the factorization~\eqref{eq:detection_GTSVD} is reused on all
$2K-1$ intervals.

Set
\[
\mu_J=\frac{2\pi}{T_{\rm det}(b_J-a_J)}.
\]
The cosine and sine coefficients of the local approximation to $q=-g''$
are then
\begin{equation}\label{eq:detection_q_coefficients}
a_{\ell,J}^{(q)}
=(\mu_J\ell)^2a_{\ell,J},
\qquad
b_{\ell,J}^{(q)}
=(\mu_J\ell)^2b_{\ell,J}.
\end{equation}
The mollified conjugate Fourier sum used for edge detection follows the
spectral-mollifier construction of Cochran, Gelb, and Wang
\cite{CochranGelbWang2013}. In the present inverse source problem, however,
the Fourier coefficients of $q=-g''$ are not directly available. They are
obtained from the noisy samples of $g$ through the local GTSVD Fourier
extension described above. This combination is consistent with the strategy
in \cite{ZhaoYuJiaDou2025}, where regularized Fourier extension was used to
recover derivative coefficients from noisy data before applying the edge
detector. We then define
\begin{equation}\label{eq:local_conjugate_sum}
\mathcal C_J(x)
=
\sum_{\ell=1}^{n_c}
\ell\widehat\sigma_{\lambda_{n_c}}(\ell)
\left[
b_{\ell,J}^{(q)}
\cos\bigl(\ell t_J(x)\bigr)
-
a_{\ell,J}^{(q)}
\sin\bigl(\ell t_J(x)\bigr)
\right],
\qquad x\in J,
\end{equation}
where
\begin{equation}\label{eq:concentration_factor}
\lambda_{n_c}
=
\frac{n_c}{\sqrt{\log n_c}},
\qquad
\widehat\sigma_{\lambda_{n_c}}(\ell)
=
\frac{1}{\sqrt\pi\lambda_{n_c}}
\exp\left[-\left(\frac{\ell}{\lambda_{n_c}}\right)^2\right].
\end{equation}
In a smooth part of $q$, $\mathcal C_J$ remains moderate, whereas an
interior jump produces a localized peak whose magnitude is related to the
jump height.

To avoid artificial endpoint responses, the coarse indicator is evaluated
only on the interior portion
\[
J^\circ
=
[a_J+\gamma H,b_J-\gamma H],
\qquad 0<\gamma<\frac12.
\]
Let $\mathcal V_c$ denote all values $|\mathcal C_J(x_i)|$ with
$J\in\mathcal P_0\cup\mathcal P_{1/2}$ and $x_i\in J^\circ$. We define the
robust coarse threshold
\begin{equation}\label{eq:coarse_detection_threshold}
\theta_c
=
\operatorname{median}(\mathcal V_c)
+
\kappa_c\operatorname{MAD}(\mathcal V_c),
\end{equation}
where the median absolute deviation is multiplied by its usual consistency
factor. Local maxima of $|\mathcal C_J|$ above $\theta_c$ are retained as
coarse candidates. Candidates from neighboring intervals and from the two staggered
partitions are merged when their distance is below $d_c$. This produces a
coarse candidate set
\[
\mathcal S_c
=
\{\xi_1^{c},\ldots,\xi_{J_c}^{c}\}.
\]
No prior information about the number of breakpoints is used.

\subsubsection{Fine localization on the original grid}

For each coarse candidate $\xi_j^c$, we select a refinement interval
$J_j^f\subset[0,\pi]$ of length $H$, centered at $\xi_j^c$ whenever the
physical boundary permits. Unlike the coarse stage, all available original
grid nodes in $J_j^f$ are used. The physical localization length is kept
unchanged; only the sampling density and, when justified by the noise level,
the local Fourier resolution are increased.

First, the coarse Fourier dimension $n_c$ is used on the full local grid to
obtain an anchor point $\xi_j^a$. The anchor is the largest conjugate-sum
peak in a fixed neighborhood of $\xi_j^c$. A local threshold of the form
$\operatorname{median}+\kappa_f\operatorname{MAD}$ is used to confirm the
candidate. This low-order full-grid step determines a stable local search
region and does not create additional breakpoints.

When the noise level permits, the local computation is repeated with a
moderately larger Fourier dimension $n_f\geq n_c$. The final point is the
dominant high-order conjugate-sum peak in
\begin{equation}\label{eq:high_order_refinement_region}
\left[
\xi_j^a-r_f,
\xi_j^a+r_f
\right],
\qquad
r_f
=
\min\left\{r_a,\frac{H}{n_f}\right\},
\end{equation}
where $r_a$ is the anchor search radius. A three-point quadratic
interpolation of the absolute peak values is used to obtain a subgrid
estimate. The role of the high-order calculation is only to sharpen the
location confirmed by the low-order anchor; it neither creates a new
candidate nor removes an already confirmed one.

After fine detections within distance $d_f$ are merged, the estimated
breakpoint set is
\begin{equation}\label{eq:estimated_breakpoint_set}
\widehat{\mathcal S}
=
\{\widehat\xi_1,\ldots,\widehat\xi_{\widehat J}\},
\qquad
0<\widehat\xi_1<\cdots<\widehat\xi_{\widehat J}<\pi.
\end{equation}
The coarse mode number is fixed throughout the detector, while $n_f$ is
increased only moderately as the noise level decreases. The specific values
used in the experiments are reported in the numerical section.

The following elementary estimate quantifies the stability of the fine
peak-location step. It is local in nature, which is consistent with the
algorithm: the coarse detector first isolates a candidate neighborhood, and
the fine stage only refines the dominant peak inside that neighborhood.

\begin{proposition}[Local stability of a detected breakpoint]
\label{prop:detector_peak_stability}
Let $U\subset J$ contain one breakpoint $\xi$, and let $\xi_J^\ast\in U$ be
an interior local maximizer of the exact finite-dimensional indicator
$s\mathcal C_J$, with $s\in\{-1,1\}$. Assume
\[
(s\mathcal C_J)'(\xi_J^\ast)=0,
\qquad
-(s\mathcal C_J)''(x)\geq\kappa_J>0,
\qquad x\in U.
\]
If $\widehat\xi$ is an interior local maximizer of
$s\mathcal C_J^\delta$ in $U$, then
\begin{equation}\label{eq:detector_peak_stability}
|\widehat\xi-\xi_J^\ast|
\leq
\frac{1}{\kappa_J}
\left\|(\mathcal C_J^\delta-\mathcal C_J)'\right\|_{L^\infty(U)}.
\end{equation}
If $\mathbf z_J^\delta$ and $\mathbf z_J$ are the regularized and exact
coefficient vectors in \eqref{eq:detection_coefficients}, respectively, then
\begin{equation}\label{eq:indicator_coefficient_stability}
\left\|(\mathcal C_J^\delta-\mathcal C_J)'\right\|_{L^\infty(U)}
\leq
C_{J,n,\sigma}\|\mathbf z_J^\delta-\mathbf z_J\|_2,
\end{equation}
where
\[
C_{J,n,\sigma}
=
\mu_J^3
\left(\sum_{\ell=1}^{n}
\ell^8\widehat\sigma_{\lambda_n}(\ell)^2\right)^{1/2}.
\]
Consequently,
\begin{equation}\label{eq:detector_total_location_error}
|\widehat\xi-\xi|
\leq
\frac{C_{J,n,\sigma}}{\kappa_J}
\|\mathbf z_J^\delta-\mathbf z_J\|_2
+|\xi_J^\ast-\xi|.
\end{equation}
The last term is the finite-dimensional detector bias.
\end{proposition}

\begin{proof}
Since $(\mathcal C_J^\delta)'(\widehat\xi)=0$, the mean-value theorem and
the curvature assumption yield
\[
\kappa_J|\widehat\xi-\xi_J^\ast|
\leq
|\mathcal C_J'(\widehat\xi)|
\leq
\| (\mathcal C_J^\delta-\mathcal C_J)'\|_{L^\infty(U)}.
\]
Differentiating \eqref{eq:local_conjugate_sum}, using
\eqref{eq:detection_q_coefficients}, and applying Cauchy--Schwarz gives
\eqref{eq:indicator_coefficient_stability}. Equation
\eqref{eq:detector_total_location_error} follows from the triangle
inequality.
\end{proof}

Proposition~\ref{prop:detector_peak_stability} separates the perturbation of
the regularized coefficients from the intrinsic finite-dimensional
localization bias. It is conditional on the coarse stage selecting the
correct one-breakpoint neighborhood; a global convergence analysis of the
detector would additionally require control of that selection step.

\subsection{Structure-aligned reconstruction of \texorpdfstring{$q$}{q}}
\label{subsec:aligned_q}

The detected breakpoints define the structural partition
\begin{equation}\label{eq:structural_partition}
0=\widehat\xi_0
<
\widehat\xi_1
<
\cdots
<
\widehat\xi_{\widehat J}
<
\widehat\xi_{\widehat J+1}=\pi,
\end{equation}
with smooth candidate components
\[
\widehat I_j
=
(\widehat\xi_{j-1},\widehat\xi_j),
\qquad
j=1,\ldots,\widehat J+1.
\]
The final reconstruction of $q=-g''$ is computed directly from the original
noisy observation $\mathbf g^\delta$, not from the conjugate-sum indicator.
No local fitting interval is permitted to cross a detected breakpoint.

Within a structural component $\widehat I_j$, we apply the adaptive
multi-interval LFE differentiation method. On a current interval $I\subset
\widehat I_j$, a fixed number $m_{\rm rec}$ of approximately equally spaced
nodes is selected, and the coefficients are computed from
\eqref{eq:GTSVD_coefficients} with $T=T_{\rm rec}$. Let $p_I^\delta$ be the
resulting approximation to $g$. Its residual on all available nodes of $I$
is
\begin{equation}\label{eq:local_residual}
r_I
=
\left(
\sum_{x_i\in I}
\left|
p_I^\delta(x_i)-g_i^\delta
\right|^2
\right)^{1/2}.
\end{equation}
The interval is accepted if
\begin{equation}\label{eq:adaptive_acceptance}
r_I
\leq
\rho\,\delta_I^{\rm all},
\qquad \rho>1,
\end{equation}
or if it contains at most $m_{
m rec}$ sampling points, or if the maximum
bisection depth $d_{\max}$ is reached. Otherwise, $I$ is bisected and the
same test is applied to its children. Such additional subdivision
is used only to resolve local oscillation inside an already identified smooth
component and has no structural interpretation.

On every accepted leaf interval, $q_I^\delta$ is evaluated analytically by
\eqref{eq:local_q_reconstruction}. Patching the local values within the
structure-aligned partition gives
\begin{equation}\label{eq:final_q}
q_{M,\delta}(x)
=
q_{j,M,\delta}(x),
\qquad
x\in\widehat I_j.
\end{equation}
At a detected breakpoint, the left and right traces are retained separately.
The value assigned at the single breakpoint itself has no effect on the
$L^2$ reconstruction.

The structural separation is essential. A local Fourier extension is never
forced to represent two different smooth branches, while further
localization remains available inside each branch when required by the data.
Thus, the same localized representation principle underlies both breakpoint
detection and numerical differentiation, whereas the detected structure
determines where local representations may not be connected.

\subsection{Recovery of the source term}
\label{subsec:source_recovery}

After $q$ has been reconstructed, the source is obtained from
\[
B_{y_0}f=q.
\]
A direct spectral inversion gives
\[
f_n
=
\frac{q_n}{1-e^{-ny_0}}.
\]
For a discontinuous or piecewise-smooth source, however, approximating the
entire source by a truncated global sine series introduces Gibbs
oscillations. We therefore use the decomposition
\[
f=q+R_{y_0}q,
\qquad
R_{y_0}=(e^{y_0\sqrt L}-I)^{-1},
\]
derived in Section~\ref{sec:direct}.

Let $\{\phi_n\}_{n\geq1}$ be the normalized sine eigenfunctions of $L$ and
define
\begin{equation}\label{eq:q_sine_coefficients}
q_n^{M,\delta}
=
\bigl(q_{M,\delta},\phi_n\bigr)
=
\sum_{j=1}^{\widehat J+1}
\int_{\widehat I_j}
q_{j,M,\delta}(x)\phi_n(x)\,dx.
\end{equation}
Since every $q_{j,M,\delta}$ is a local trigonometric polynomial, the
integrals in \eqref{eq:q_sine_coefficients} can be evaluated analytically or
by sufficiently accurate quadrature on the individual structural
components.

Define the truncated smooth correction
\begin{equation}\label{eq:smooth_correction}
R_{y_0,N}q_{M,\delta}
=
\sum_{n=1}^{N}
\frac{q_n^{M,\delta}}{e^{ny_0}-1}
\phi_n.
\end{equation}
The final source reconstruction is
\begin{equation}\label{eq:final_source}
f_{M,N,\delta}(x)
=
q_{M,\delta}(x)
+
\sum_{n=1}^{N}
\frac{q_n^{M,\delta}}{e^{ny_0}-1}
\phi_n(x).
\end{equation}
The first term in \eqref{eq:final_source} retains the detected jumps and the
piecewise-smooth structure. The second term is globally smooth, and its
coefficients decay exponentially with $n$. Therefore, only a moderate number
of global sine modes is needed for the correction.

The following estimate separates the differentiation error from the
truncation error of the smooth correction.

\begin{proposition}\label{prop:source_error}
Let
\[
f=(I+R_{y_0})q
\]
and let $f_{M,N,\delta}$ be defined by \eqref{eq:final_source}. Then
\begin{equation}\label{eq:source_error_bound}
\|f-f_{M,N,\delta}\|
\leq
\frac{1}{1-e^{-y_0}}
\|q-q_{M,\delta}\|
+
\frac{1}{e^{(N+1)y_0}-1}\|q\|.
\end{equation}
\end{proposition}

\begin{proof}
Let $R_{y_0,N}$ denote the spectral truncation in
\eqref{eq:smooth_correction}. Then
\[
f-f_{M,N,\delta}
=
(I+R_{y_0,N})(q-q_{M,\delta})
+
(R_{y_0}-R_{y_0,N})q.
\]
Since
\[
\|I+R_{y_0,N}\|
\leq
1+\frac{1}{e^{y_0}-1}
=
\frac{1}{1-e^{-y_0}},
\]
and
\[
\|R_{y_0}-R_{y_0,N}\|
\leq
\frac{1}{e^{(N+1)y_0}-1},
\]
the result follows.
\end{proof}

We next combine Proposition~\ref{prop:source_error} with the piecewise
Fourier-extension differentiation estimate. The following statement isolates
the regularization error by assuming that the exact structural partition is
known. As usual, the spatial discretization and finite-dimensional
approximation are taken sufficiently fine that their errors are of lower
order than the regularization error.

\begin{theorem}[Convergence for exact breakpoints]
\label{thm:exact_breakpoint_convergence}
Let
\[
0=\xi_0<\xi_1<\cdots<\xi_J<\xi_{J+1}=\pi
\]
be the exact breakpoint partition, and assume that
\[
g|_{I_j}\in H^{s_j}(I_j),
\qquad
I_j=(\xi_{j-1},\xi_j),
\qquad
s_j\geq3.
\]
Set
\[
\bar s=\min_{1\leq j\leq J+1}s_j,
\qquad
\alpha_{\rm diff}=\frac{\bar s-2}{\bar s}.
\]
Let $q_{\delta}^{\mathcal S}$ be the structure-aligned regularized Fourier
extension approximation of $q=-g''$ on the exact partition, with the local
regularization parameters chosen by a discrepancy principle. Then
\begin{equation}\label{eq:exact_partition_q_rate}
\|q-q_{\delta}^{\mathcal S}\|
\leq
C_{\rm diff}\delta^{\alpha_{\rm diff}}.
\end{equation}
Define
\[
f_{\delta,N}^{\mathcal S}
=
(I+R_{y_0,N})q_{\delta}^{\mathcal S}.
\]
Then
\begin{equation}\label{eq:exact_partition_source_rate}
\|f-f_{\delta,N}^{\mathcal S}\|
\leq
\frac{C_{\rm diff}}{1-e^{-y_0}}
\delta^{\alpha_{\rm diff}}
+
\frac{\|q\|}{e^{(N+1)y_0}-1}.
\end{equation}
In particular, if $N=N(\delta)$ is chosen so that
\[
\frac{1}{e^{(N(\delta)+1)y_0}-1}
=
\mathcal O\!\left(\delta^{\alpha_{\rm diff}}\right),
\]
then
\begin{equation}\label{eq:exact_partition_final_rate}
\|f-f_{\delta,N(\delta)}^{\mathcal S}\|
=
\mathcal O\!\left(
\delta^{(\bar s-2)/\bar s}
\right).
\end{equation}
\end{theorem}

\begin{proof}
Estimate~\eqref{eq:exact_partition_q_rate} is the $k=2$ case of the
piecewise Fourier-extension differentiation result in
\cite{ZhaoYuJiaDou2025}. Applying
Proposition~\ref{prop:source_error} with
$q_{M,\delta}=q_{\delta}^{\mathcal S}$ gives
\eqref{eq:exact_partition_source_rate}. The prescribed choice of
$N(\delta)$ makes the spectral truncation term no larger than the
differentiation term and yields \eqref{eq:exact_partition_final_rate}.
\end{proof}

\begin{remark}[Rate in terms of source regularity]
If $f|_{I_j}\in H^{r_j}(I_j)$, then $q=(I-P_{y_0})f$ has the same local
regularity and $g''=-q$ gives $g|_{I_j}\in H^{r_j+2}(I_j)$. Thus, with
$\bar r=\min_j r_j$, the exact-partition rate can be written as
$\mathcal O(\delta^{\bar r/(\bar r+2)})$.
\end{remark}

\begin{remark}[Stability under small interface displacements]
\label{rem:breakpoint_location_error}
Theorem~\ref{thm:exact_breakpoint_convergence} describes the differentiation
and correction errors on the exact structural partition. For a computed
partition, suppose that the detected breakpoints can be matched in order with
the exact ones and set
\[
\varepsilon_{\rm bp}
=
\max_{1\leq j\leq J}
|\widehat\xi_j-\xi_j|.
\]
The two partitions differ only on
\[
\Omega_{\rm bp}
=
\bigcup_{j=1}^{J}
\left[
\min\{\xi_j,\widehat\xi_j\},
\max\{\xi_j,\widehat\xi_j\}
\right],
\qquad
|\Omega_{\rm bp}|
\leq
J\varepsilon_{\rm bp}.
\]
If the exact-partition estimate holds on the common smooth region and the
exact and reconstructed branches remain uniformly bounded near the
interfaces, then
\begin{equation}\label{eq:breakpoint_perturbation_q}
\|q-q_{\delta}^{\widehat{\mathcal S}}\|
\leq
C_{\rm diff}\delta^{\alpha_{\rm diff}}
+
C_{\rm bp}\varepsilon_{\rm bp}^{1/2},
\end{equation}
and hence
\begin{equation}\label{eq:breakpoint_perturbation_f}
\|f-f_{\delta,N}^{\widehat{\mathcal S}}\|
\leq
\frac{
C_{\rm diff}\delta^{\alpha_{\rm diff}}
+
C_{\rm bp}\varepsilon_{\rm bp}^{1/2}}
{1-e^{-y_0}}
+
\frac{\|q\|}{e^{(N+1)y_0}-1}.
\end{equation}
Here $C_{\rm bp}$ depends on the jump magnitudes and on local bounds of the
neighboring smooth branches. Together with
Proposition~\ref{prop:detector_peak_stability}, this shows explicitly how a
perturbation of the local detector coefficients produces a localized
interface mismatch and how that mismatch is transferred through the
well-posed source correction. In particular, vanishing coefficient and peak
perturbations imply a vanishing partition contribution.
\end{remark}

Thus, once a structure-aligned approximation of $q$ is available, the final
source recovery adds only a bounded amplification, an exponentially small
spectral truncation term, and, for detected rather than exact breakpoints, a
localized mismatch contribution.

\subsection{Summary of the reconstruction procedure}
\label{subsec:algorithm_summary}

The complete reconstruction method is summarized in
Algorithm~\ref{alg:source_reconstruction}.

\begin{algorithm}[H]
\caption{Localized reconstruction of a piecewise-smooth source}
\label{alg:source_reconstruction}
\begin{algorithmic}[1]
\REQUIRE Noisy data $\mathbf g^\delta$, $y_0$, $H$, a noise estimate, and local parameters.

\STATE Construct the primary and staggered local partitions
$\mathcal P_0$ and $\mathcal P_{1/2}$.

\STATE Precompute the detector GTSVD and evaluate the local indicator
\eqref{eq:local_conjugate_sum} on $\mathcal P_0\cup\mathcal P_{1/2}$.

\STATE Extract significant local peaks using the robust threshold
\eqref{eq:coarse_detection_threshold}, and merge repeated candidates from
the two staggered partitions.

\STATE Confirm each candidate on the fine grid and refine its location with the noise-adaptive Fourier dimension.

\STATE Construct the structural partition from the detected breakpoint set
$\widehat{\mathcal S}$.

\STATE Apply adaptive structure-aligned LFE differentiation to the original data and obtain $q_{M,\delta}$.

\STATE Compute the piecewise sine coefficients $q_n^{M,\delta}$ from
\eqref{eq:q_sine_coefficients}.

\STATE Recover the source by
\[
f_{M,N,\delta}
=
q_{M,\delta}
+
\sum_{n=1}^{N}
\frac{q_n^{M,\delta}}{e^{ny_0}-1}\phi_n.
\]

\RETURN The estimated breakpoint set $\widehat{\mathcal S}$ and the
reconstructed source $f_{M,N,\delta}$.
\end{algorithmic}
\end{algorithm}

\section{Numerical experiments}
\label{sec:numerics}

We assess breakpoint localization, reconstruction of $q=-g''$, final
source recovery, and robustness to the noise distribution. The comparisons
use Yang's truncated Fourier method \cite{Yang2011} and full-grid first-order
TV regularization \cite{RudinOsherFatemi1992,AcarVogel1994,BenningBurger2018}.

\subsection{Experimental setting and test sources}
\label{subsec:numerical_setting}

The observation height is fixed at
\[
 y_0=0.7.
\]
We use $M_x=2305$ uniformly distributed nodes
\[
 x_i=\frac{i\pi}{M_x-1},
 \qquad i=0,1,\ldots,M_x-1,
\]
so that $h=\pi/2304$.  The test sources have the general form
\begin{equation}
\label{eq:test_source_general}
 f(x)=b_0+b_1\frac{x}{\pi}+A\sin x
 +\sum_{j=1}^{J}d_j\mathcal H(x-\xi_j),
\end{equation}
where $\mathcal H$ is the Heaviside function, $\xi_j$ are the exact
breakpoints, and $d_j$ are the corresponding jump heights.  The value
assigned at a breakpoint does not affect any of the reported $L^2$ errors.

To generate the observation without introducing an additional PDE
discretization error, we use the exact sine-series representation of the
forward solution.  With the convention
\[
 f(x)\sim\sum_{n=1}^{\infty}\widehat f_n\sin(nx),
 \qquad
 \widehat f_n=\frac{2}{\pi}\int_0^\pi f(x)\sin(nx)\,dx,
\]
the coefficients of \eqref{eq:test_source_general} are
\begin{equation}
\label{eq:test_source_coefficients}
\widehat f_n
=
\frac{2}{\pi}
\left[
 \frac{b_0\bigl(1-(-1)^n\bigr)}{n}
 -\frac{b_1(-1)^n}{n}
 +\sum_{j=1}^{J}
 \frac{d_j\bigl(\cos(n\xi_j)-(-1)^n\bigr)}{n}
\right]
+A\,\delta_{n1}.
\end{equation}
The exact observation and the corresponding transformed source are evaluated
from
\begin{align}
 g(x)
 &=
 \sum_{n=1}^{N_{\rm ex}}
 \frac{1-e^{-ny_0}}{n^2}\widehat f_n\sin(nx),
 \label{eq:exact_observation_series}\\
 q(x)
 &=
 f(x)-
 \sum_{n=1}^{N_{\rm ex}}
 e^{-ny_0}\widehat f_n\sin(nx),
 \label{eq:exact_q_series}
\end{align}
with $N_{\rm ex}=2400$.

The representative figures and the deterministic comparison table use
bounded uniform perturbations
\begin{equation}
\label{eq:numerical_noise}
 g_i^\delta=g(x_i)+\delta\zeta_i,
 \qquad
 \zeta_i\sim\operatorname{Unif}[-1,1],
\end{equation}
where the endpoint perturbations are set to zero.  The same normalized noise
realization is rescaled at all noise levels, so that changes with respect to
$\delta$ can be compared directly.  Since the standard deviation of
$\delta\zeta_i$ is $\delta/\sqrt{3}$, the discrepancy targets are scaled
accordingly.

To assess robustness with respect to the noise distribution, we also use
RMS-matched Gaussian perturbations
\begin{equation}
\label{eq:gaussian_noise}
 g_i^\delta=g(x_i)+\frac{\delta}{\sqrt{3}}\,\eta_i,
 \qquad
 \eta_i\sim\mathcal N(0,1).
\end{equation}
For each source and each noise level, 20 independent Gaussian realizations
are generated.  Because the representative reconstructions under
\eqref{eq:numerical_noise} and \eqref{eq:gaussian_noise} are visually and
quantitatively similar, only the uniform-noise results are plotted; the
Gaussian experiments are summarized statistically in
Table~\ref{tab:gaussian_robustness}.

Three sources are considered:
\begin{align}
 f^{(1)}(x)
 &=0.8\sin x
 +2.5\mathcal H(x-0.85)
 -2.5\mathcal H(x-2.30),
 \label{eq:test_source_1}\\
 f^{(2)}(x)
 &=0.6+0.3\frac{x}{\pi}+0.8\sin x
 +2.5\mathcal H(x-0.85)
 -2.5\mathcal H(x-2.30),
 \label{eq:test_source_2}\\
 f^{(3)}(x)
 &=0.6-0.2\frac{x}{\pi}+0.9\sin x
 +1.6\mathcal H(x-0.70)
 -0.7\mathcal H(x-1.55)
 +0.3\mathcal H(x-2.40).
 \label{eq:test_source_3}
\end{align}
The first two sources contain two jumps of equal magnitude.  They differ in
that $f^{(1)}(0)=f^{(1)}(\pi)=0$, whereas
$f^{(2)}(0)=0.6$ and $f^{(2)}(\pi)=0.9$.  The third source contains three
unequal jumps; the smallest jump has magnitude $0.3$, approximately $10\%$
of $\max_x|f^{(3)}(x)|$.  No information about the number or locations of
the breakpoints is supplied to any reconstruction method.

The parameters are summarized in Table~\ref{tab:numerical_parameters}.
The detector uses four primary intervals of length $H=\pi/4$ and three
half-shifted intervals. Each coarse fit uses $m_c=19$ nodes and $n_c=9$
modes. Fine localization uses all $577$ nodes in an interval of length $H$
and
\begin{equation}\label{eq:fine_mode_rule_numerics}
n_f=
\begin{cases}
9, & \delta\geq10^{-3},\\
12, & 10^{-4}\leq\delta<10^{-3},\\
15, & \delta<10^{-4}.
\end{cases}
\end{equation}
The geometry and thresholds are fixed across all sources and noise
realizations.

\begin{table}[!t]
\centering
\caption{Parameters used in the numerical experiments.}
\label{tab:numerical_parameters}
\scriptsize
\setlength{\tabcolsep}{4pt}
\begin{tabular}{lll}
\toprule
Stage & Parameter & Value \\
\midrule
Forward data
& $(M_x,y_0,N_{\rm ex})$
& $(2305,0.7,2400)$ \\
Coarse detection
& $(H,n_c,m_c,T_{\rm det})$
& $(\pi/4,9,19,2)$ \\
& $(\beta,\tau_{\rm det},\varepsilon_{\rm sv})$
& $(0.49,1.05,10^{-7})$ \\
& $(\gamma,\kappa_c,d_c)$
& $(0.15,2.0,H/4)$ \\
Fine localization
& $(r_a,\kappa_f,d_f)$
& $(H/4,1.5,H/8)$ \\
& nodes and $n_f$
& $577$ and \eqref{eq:fine_mode_rule_numerics} \\
Differentiation
& $(n,m_{\rm rec},T_{\rm rec},\tau_{\rm rec},\rho)$
& $(9,19,6,1.10,2.00)$ \\
& $(\varepsilon_{\rm sv},d_{\max})$
& $(10^{-10},14)$ \\
Source correction
& $N_{\rm corr}$
& $100$ \\
Yang truncation
& $(N_{\max},C_{\rm DP})$
& $(300,1.10)$ \\
Full-grid TV
& $(C_{\rm TV},N_{\rm ADMM},N_{\rm out})$
& $(1.10,3000,10)$ \\
Gaussian tests
& realizations per case
& $20$ \\
\bottomrule
\end{tabular}
\end{table}

For Yang's comparison method, the sine coefficients of the noisy observation
are computed by the composite trapezoidal rule, and the source is
reconstructed by
\begin{equation}
\label{eq:yang_numerical_formula}
 f_{N,\mathrm{Y}}^\delta(x)
 =
 \sum_{n=1}^{N}
 \frac{n^2}{1-e^{-ny_0}}
 \widehat g_n^\delta\sin(nx).
\end{equation}
The cutoff $N$ is the smallest integer satisfying
\begin{equation}
\label{eq:yang_discrepancy_numerics}
 \left\|
 \sum_{n=1}^{N}\widehat g_n^\delta\sin(nx_i)
 -g_i^\delta
 \right\|_2
 \leq
 C_{\rm DP}\,\delta\sqrt{\frac{M_x}{3}},
 \qquad C_{\rm DP}=1.10,
\end{equation}
subject to $N\leq N_{\max}=300$.

The TV comparison is performed directly on the full set of interior grid
values.  Let $Q$ denote the orthogonal discrete sine transform and set
\[
 A_h
 =
 Q\operatorname{diag}\left(
 \frac{1-e^{-ny_0}}{n^2}
 \right)Q^{\mathsf T}.
\]
The TV reconstruction is defined by the standard quadratic data-fidelity
plus first-order bounded-variation penalty
\cite{AcarVogel1994,BenningBurger2018}:
\begin{equation}
\label{eq:tv_numerical_model}
 f_{\alpha,\mathrm{TV}}^\delta
 =
 \arg\min_{v\in\mathbb R^{M_x-2}}
 \left\{
 \frac12\|A_hv-g_{\rm int}^\delta\|_2^2
 +\alpha\|Dv\|_1
 \right\},
\end{equation}
where $D$ is the first-difference matrix on adjacent interior nodes.  No
difference to a prescribed endpoint value is included, so nonzero endpoint
traces are not penalized artificially.  All interior sine modes are retained,
and neither exact nor detected breakpoint information enters
\eqref{eq:tv_numerical_model}.  The parameter $\alpha$ is selected by the Morozov principle with
$C_{\rm TV}=1.10$. Starting from $\alpha_0=\max\{10^{-14},\delta\}$, a
doubling/halving bracket and at most ten bisection steps are used; the
selected value is the largest tested $\alpha$ whose residual does not exceed
the discrepancy radius, with a $1.5\%$ stopping tolerance. The split
Bregman/ADMM iteration \cite{GoldsteinOsher2009} uses at most $3000$ steps.
Its penalty is initialized by
$\varrho_0=\min\{10^{-1},\max\{10^{-7},5\alpha\}\}$ and updated every
$100$ steps when the primal/dual residual ratio exceeds $10$. The absolute,
relative, and relative-change tolerances are $10^{-7}$, $2\times10^{-6}$,
and $10^{-6}$, respectively. Source updates are evaluated in DST coordinates,
so no dense matrix is formed.

For a set $\Omega\subset[0,\pi]$, the relative error is measured by
\begin{equation}
\label{eq:relative_error_numerics}
 E_{\Omega}(v_{\rm num})
 =
 \frac{\|v_{\rm num}-v\|_{L^2(\Omega)}}
 {\|v\|_{L^2(\Omega)}}.
\end{equation}
In addition to the global error, we report the error on the smooth region
\begin{equation}
\label{eq:smooth_region_numerics}
 \Omega_{\rm sm}
 =
 [0.02,\pi-0.02]
 \setminus
 \bigcup_{j=1}^{J}
 (\xi_j-0.03,\xi_j+0.03).
\end{equation}
The exact breakpoint locations in \eqref{eq:smooth_region_numerics} are used
only to assess the numerical error and are not used by the algorithm.  When
the detected and exact breakpoints are matched in increasing order, the
localization error is
\begin{equation}
\label{eq:breakpoint_error_numerics}
 E_{\rm bp}
 =
 \max_{1\leq j\leq J}
 |\widehat\xi_j-\xi_j|.
\end{equation}

\subsection{Breakpoint detection and reconstruction of
\texorpdfstring{$q=-g''$}{q=-g''}}
\label{subsec:numerical_detection_q}

Table~\ref{tab:detection_q_results} summarizes the uniform-noise results.
All true breakpoints are recovered in the nine tests without false
detections. For the two strong-jump sources, the localization error decreases
from about $5\times10^{-3}$ at $\delta=10^{-3}$ to $10^{-4}$--$10^{-5}$ at
lower noise. For the unequal-jump source, all three interfaces, including the
jump of height $0.3$, are detected at $\delta=10^{-4}$; its maximum error
decreases from $1.607\times10^{-3}$ to $3.269\times10^{-4}$.

\begin{table}[htbp]
\centering
\caption{Breakpoint detection and reconstruction errors for $q=-g''$ under
the representative uniform perturbation.  Here ``det.'' denotes the number
of matched detections divided by the number of true breakpoints.}
\label{tab:detection_q_results}
\scriptsize
\begin{tabular}{ccccccc}
\toprule
Source & $\delta$ & $n_f$ & det. & false & $E_{\rm bp}$
& $E_{[0,\pi]}(q)$ / $E_{\Omega_{\rm sm}}(q)$ \\
\midrule
$f^{(1)}$ & $10^{-3}$ & 9  & $2/2$ & 0 & $4.872\times10^{-3}$
& $1.388\times10^{-1}$ / $1.102\times10^{-1}$ \\
$f^{(1)}$ & $10^{-4}$ & 12 & $2/2$ & 0 & $1.527\times10^{-4}$
& $1.817\times10^{-2}$ / $1.475\times10^{-2}$ \\
$f^{(1)}$ & $10^{-5}$ & 15 & $2/2$ & 0 & $5.134\times10^{-5}$
& $1.393\times10^{-3}$ / $9.688\times10^{-4}$ \\
\midrule
$f^{(2)}$ & $10^{-3}$ & 9  & $2/2$ & 0 & $5.452\times10^{-3}$
& $9.242\times10^{-2}$ / $6.168\times10^{-2}$ \\
$f^{(2)}$ & $10^{-4}$ & 12 & $2/2$ & 0 & $1.443\times10^{-4}$
& $1.861\times10^{-2}$ / $1.321\times10^{-2}$ \\
$f^{(2)}$ & $10^{-5}$ & 15 & $2/2$ & 0 & $8.148\times10^{-5}$
& $1.161\times10^{-3}$ / $6.961\times10^{-4}$ \\
\midrule
$f^{(3)}$ & $10^{-4}$ & 12 & $3/3$ & 0 & $1.607\times10^{-3}$
& $3.655\times10^{-2}$ / $2.472\times10^{-2}$ \\
$f^{(3)}$ & $10^{-5}$ & 15 & $3/3$ & 0 & $4.655\times10^{-4}$
& $3.644\times10^{-3}$ / $2.501\times10^{-3}$ \\
$f^{(3)}$ & $10^{-6}$ & 15 & $3/3$ & 0 & $3.269\times10^{-4}$
& $9.436\times10^{-4}$ / $4.437\times10^{-4}$ \\
\bottomrule
\end{tabular}
\end{table}

The $q$ errors decrease consistently with the noise level. Their global and
smooth-region difference is explained by the
$\mathcal O(\varepsilon_{\rm bp}^{1/2})$ mismatch term in
Remark~\ref{rem:breakpoint_location_error}. The nonzero endpoint values of
$f^{(2)}$ cause no loss of stability because each branch is reconstructed
locally. Figure~\ref{fig:source3_structure} shows the hardest detection case;
the corresponding results for $f^{(1)}$ and $f^{(2)}$ are fully summarized
in Table~\ref{tab:detection_q_results}. At low noise, occasional saturation
of the finite GTSVD space reflects an approximation floor rather than the
inclusion of unreliable singular components.

\begin{figure}[!t]
\centering
\includegraphics[width=0.4\textwidth]{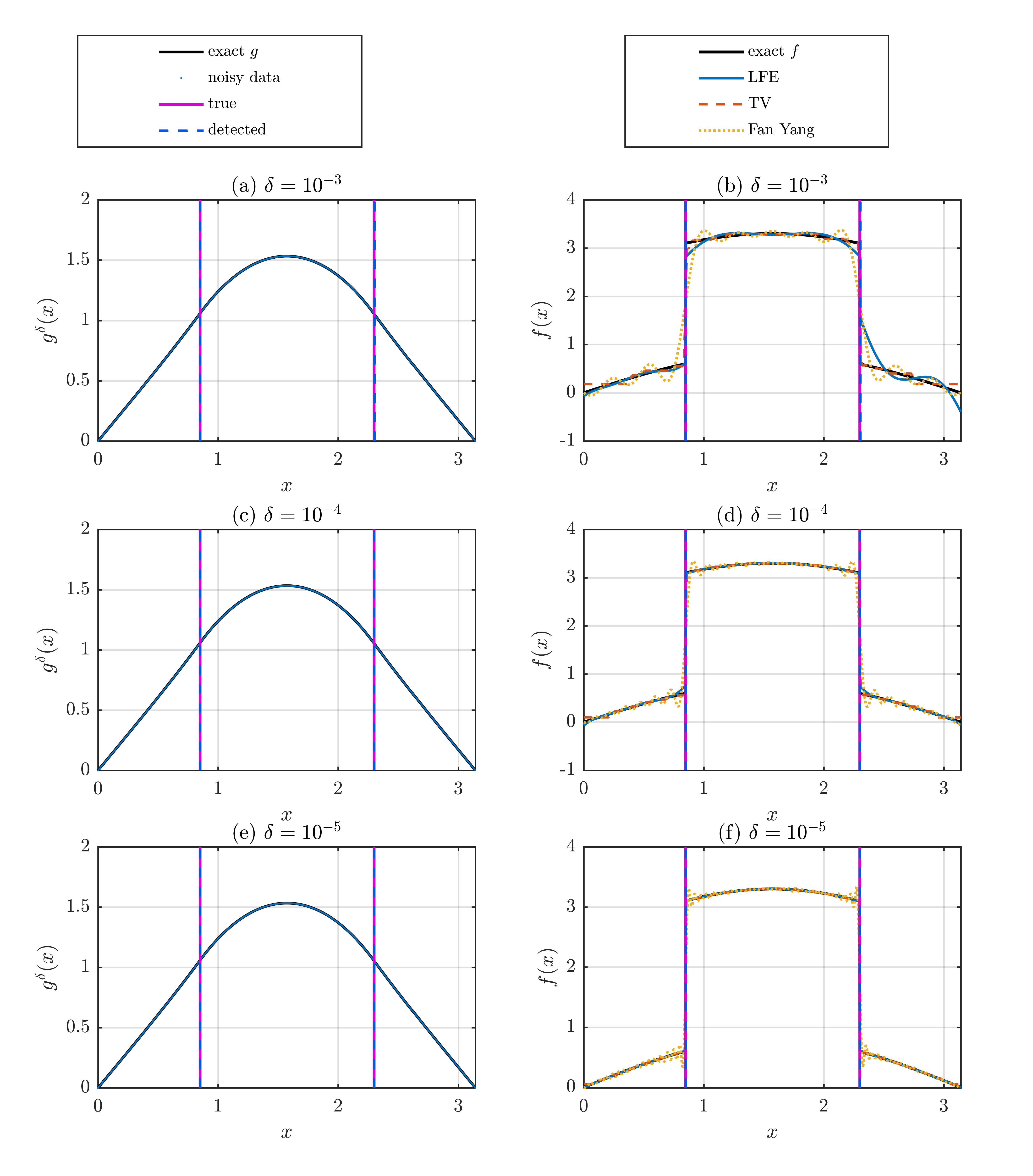}
\caption{Observation data and inverse-source reconstructions for
$f^{(1)}$.  The left column shows the exact and noisy observations; the right
column compares the proposed piecewise LFE reconstruction with full-grid TV
regularization and Yang's Fourier truncation method.  The rows correspond to
$\delta=10^{-3},10^{-4},10^{-5}$.}
\label{fig:source1_inverse}
\end{figure}

\begin{figure}[!t]
\centering
\includegraphics[width=0.4\textwidth]{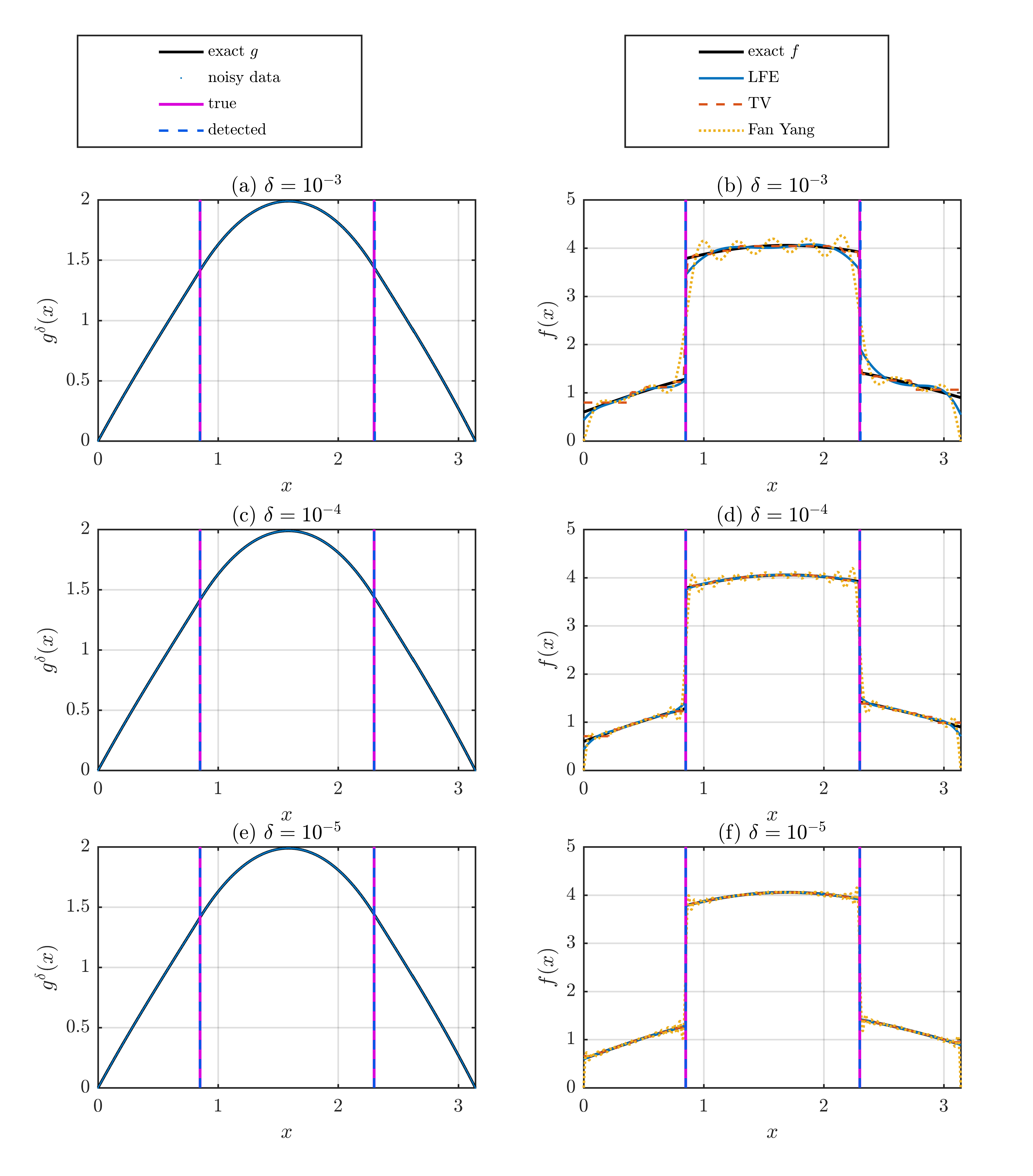}
\caption{Observation data and inverse-source reconstructions for
$f^{(2)}$.  The layout is the same as in Fig.~\ref{fig:source1_inverse}, and
the rows correspond to $\delta=10^{-3},10^{-4},10^{-5}$.}
\label{fig:source2_inverse}
\end{figure}

\begin{figure}[!t]
\centering
\includegraphics[width=0.4\textwidth]{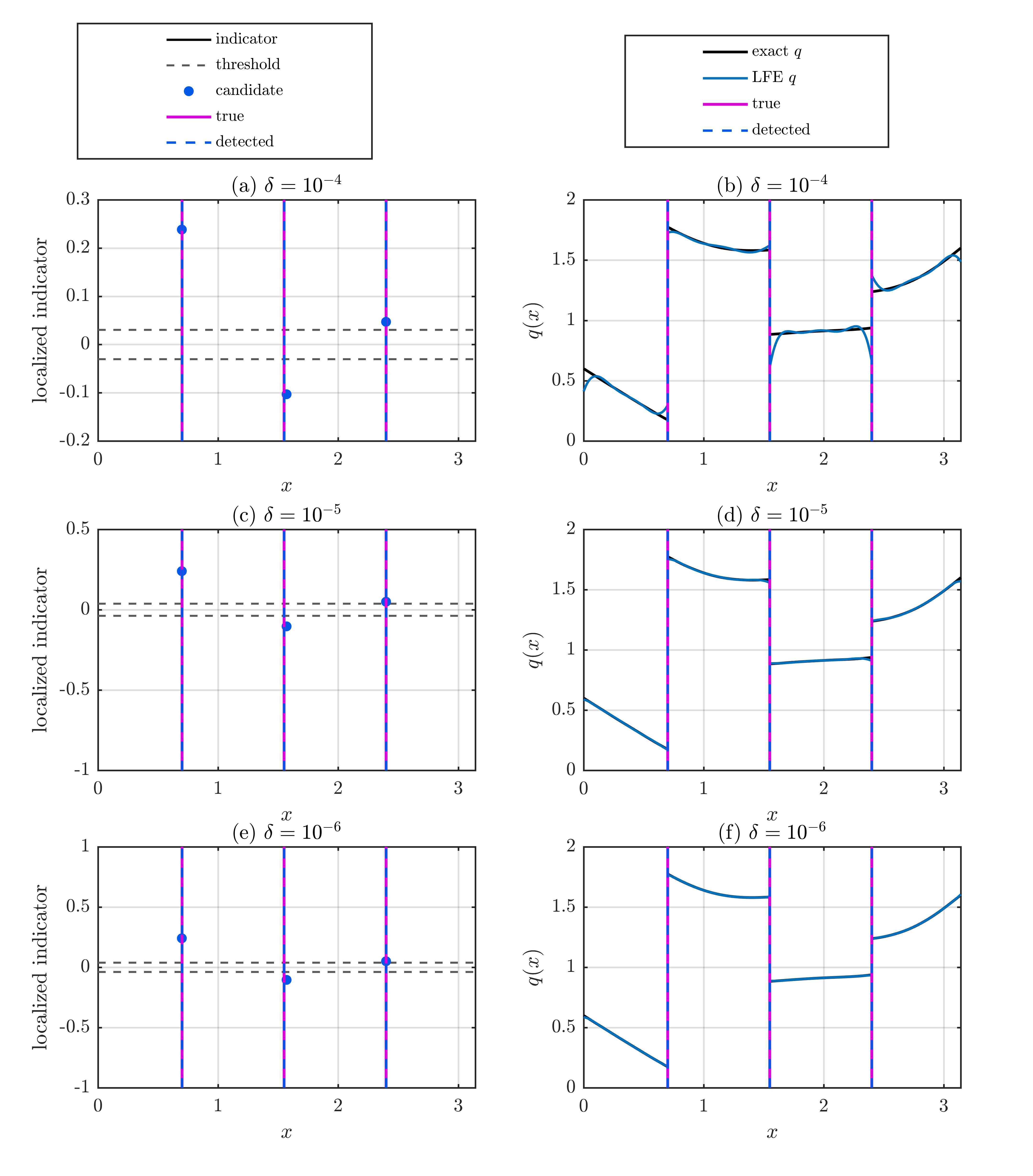}
\caption{Breakpoint detection and reconstruction of $q=-g''$ for the
unequal-jump source $f^{(3)}$, including the weak jump of height $0.3$.  The
rows correspond to $\delta=10^{-4},10^{-5},10^{-6}$.}
\label{fig:source3_structure}
\end{figure}

\begin{figure}[!t]
\centering
\includegraphics[width=0.4\textwidth]{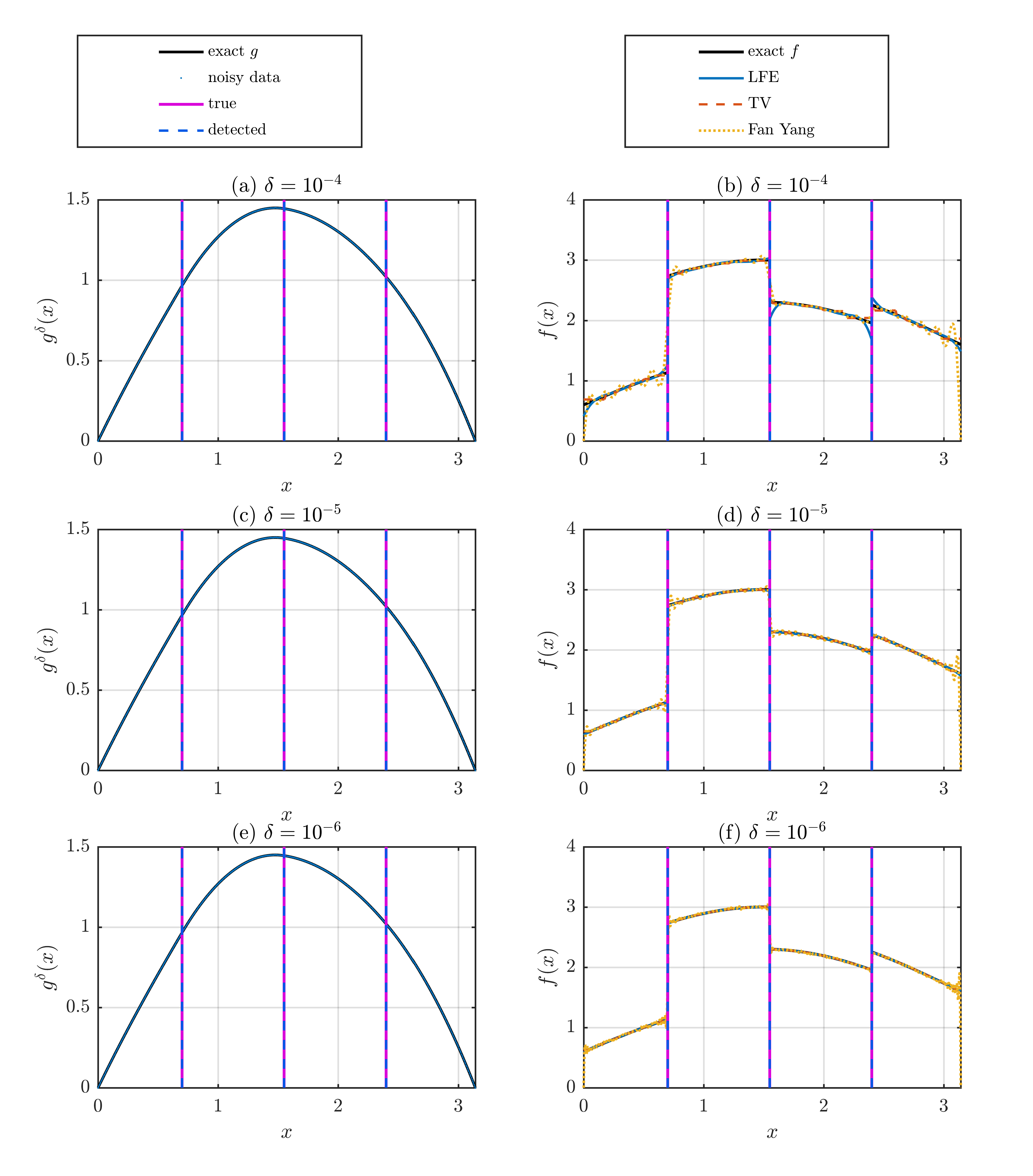}
\caption{Observation data and inverse-source reconstructions for
$f^{(3)}$.  The right column compares the piecewise LFE, full-grid TV, and
Yang reconstructions.  The rows correspond to
$\delta=10^{-4},10^{-5},10^{-6}$.}
\label{fig:source3_inverse}
\end{figure}

\subsection{Source reconstruction and comparison methods}
\label{subsec:numerical_source_comparison}

Table~\ref{tab:source_comparison_results} shows a clear regime change.
At the largest noise levels, TV is slightly more accurate than LFE; for
$f^{(3)}$ at $\delta=10^{-4}$ their global errors are
$1.994\times10^{-2}$ and $2.047\times10^{-2}$. Once the breakpoints are
accurately localized, LFE exploits the smooth branches and converges much
faster. At $\delta=10^{-5}$, its errors for $f^{(1)}$ and $f^{(2)}$ are about
$29$ times smaller than TV and $57$--$61$ times smaller than truncated
Fourier inversion. For $f^{(3)}$ at $\delta=10^{-6}$, the three errors are
$5.359\times10^{-4}$, $8.593\times10^{-3}$, and $2.659\times10^{-2}$.

\begin{table*}[!t]
\centering
\caption{Relative source reconstruction errors under the representative
uniform perturbation.  Each entry in an error column gives
$E_{[0,\pi]}/E_{\Omega_{\rm sm}}$.  The column $N_{\rm Y}$ is the cutoff
selected for Yang's method.}
\label{tab:source_comparison_results}
\scriptsize
\setlength{\tabcolsep}{5pt}
\begin{tabular}{cccccc}
\toprule
Source & $\delta$
& $E(f_{\rm LFE})$
& $E(f_{\rm TV})$
& $E(f_{\rm Y})$
& $N_{\rm Y}$ \\
\midrule
$f^{(1)}$ & $10^{-3}$
& $8.191\times10^{-2}/6.553\times10^{-2}$
& $4.907\times10^{-2}/2.201\times10^{-2}$
& $1.120\times10^{-1}/9.453\times10^{-2}$ & 19 \\
$f^{(1)}$ & $10^{-4}$
& $1.079\times10^{-2}/8.834\times10^{-3}$
& $3.490\times10^{-2}/9.033\times10^{-3}$
& $7.183\times10^{-2}/3.126\times10^{-2}$ & 50 \\
$f^{(1)}$ & $10^{-5}$
& $7.994\times10^{-4}/5.585\times10^{-4}$
& $2.334\times10^{-2}/3.578\times10^{-3}$
& $4.586\times10^{-2}/2.054\times10^{-2}$ & 123 \\
\midrule
$f^{(2)}$ & $10^{-3}$
& $5.048\times10^{-2}/3.395\times10^{-2}$
& $3.835\times10^{-2}/1.700\times10^{-2}$
& $9.435\times10^{-2}/7.700\times10^{-2}$ & 21 \\
$f^{(2)}$ & $10^{-4}$
& $1.014\times10^{-2}/7.213\times10^{-3}$
& $2.743\times10^{-2}/7.131\times10^{-3}$
& $6.037\times10^{-2}/2.562\times10^{-2}$ & 51 \\
$f^{(2)}$ & $10^{-5}$
& $6.336\times10^{-4}/3.886\times10^{-4}$
& $1.828\times10^{-2}/2.823\times10^{-3}$
& $3.846\times10^{-2}/1.580\times10^{-2}$ & 127 \\
\midrule
$f^{(3)}$ & $10^{-4}$
& $2.047\times10^{-2}/1.397\times10^{-2}$
& $1.994\times10^{-2}/1.103\times10^{-2}$
& $6.367\times10^{-2}/2.906\times10^{-2}$ & 47 \\
$f^{(3)}$ & $10^{-5}$
& $2.104\times10^{-3}/1.478\times10^{-3}$
& $1.268\times10^{-2}/4.457\times10^{-3}$
& $4.116\times10^{-2}/1.699\times10^{-2}$ & 115 \\
$f^{(3)}$ & $10^{-6}$
& $5.359\times10^{-4}/2.710\times10^{-4}$
& $8.593\times10^{-3}/1.806\times10^{-3}$
& $2.659\times10^{-2}/7.532\times10^{-3}$ & 284 \\
\bottomrule
\end{tabular}
\end{table*}

The methods are therefore complementary. TV is robust at high noise but
retains a first-order bias on nonconstant branches; the global sine method is
limited by Gibbs oscillations. LFE preserves the piecewise leading term
$q_{M,\delta}$ and expands only the smooth correction,
\[
R_{y_0}q_{M,\delta}
=
\sum_{n=1}^{N_{\rm corr}}
\frac{q_n^{M,\delta}}{e^{ny_0}-1}\phi_n,
\]
whose coefficients decay exponentially. Figures~\ref{fig:source1_inverse},
\ref{fig:source2_inverse}, and \ref{fig:source3_inverse} illustrate the
resulting low-noise advantage.

\subsection{Robustness under Gaussian perturbations}
\label{subsec:gaussian_robustness}

Table~\ref{tab:gaussian_robustness} reports 20 independent RMS-matched
Gaussian trials per case. A trial is successful when the correct number of
breakpoints is recovered without false detections. For
$\delta\leq10^{-4}$, the success rate is $100\%$, including the weak jump in
$f^{(3)}$, and the mean errors are close to their uniform-noise counterparts.
At $\delta=10^{-3}$, the rates decrease to $90\%$ and $85\%$ for the first
two sources, and occasional detection or differentiation outliers enlarge
the all-trial errors. This identifies $\delta=10^{-3}$ as the practical
stability boundary of the present detection--differentiation pipeline.

\begin{table*}[!t]
\centering
\caption{Robustness under RMS-matched Gaussian noise over 20 independent
realizations.  Errors are reported as sample mean, with sample standard
deviation in parentheses.}
\label{tab:gaussian_robustness}
\scriptsize
\setlength{\tabcolsep}{2pt}
\begin{tabular}{llccccc}
\toprule
Example & $\delta$ & Success & $E_{\rm bp}$
& $E_{\rm LFE}^{\rm all}$
& $E_{\rm LFE}^{\rm succ}$
& $E_{\rm Y}$ \\
\midrule
Ex.~1 & $10^{-3}$ & $90.0\%$
& $2.34\times10^{-3}\;(1.1\times10^{-3})$
& $1.45\times10^{-1}\;(2.5\times10^{-1})$
& $1.20\times10^{-1}\;(2.4\times10^{-1})$
& $1.11\times10^{-1}\;(2.0\times10^{-3})$ \\
Ex.~1 & $10^{-4}$ & $100.0\%$
& $3.66\times10^{-4}\;(1.9\times10^{-4})$
& $1.17\times10^{-2}\;(5.7\times10^{-3})$
& $1.17\times10^{-2}\;(5.7\times10^{-3})$
& $7.13\times10^{-2}\;(1.0\times10^{-3})$ \\
Ex.~1 & $10^{-5}$ & $100.0\%$
& $4.22\times10^{-5}\;(1.9\times10^{-5})$
& $9.36\times10^{-4}\;(4.4\times10^{-4})$
& $9.36\times10^{-4}\;(4.4\times10^{-4})$
& $4.57\times10^{-2}\;(4.9\times10^{-4})$ \\
\midrule
Ex.~2 & $10^{-3}$ & $85.0\%$
& $3.22\times10^{-3}\;(9.7\times10^{-4})$
& $1.75\times10^{-1}\;(5.2\times10^{-1})$
& $4.27\times10^{-2}\;(7.7\times10^{-3})$
& $9.03\times10^{-2}\;(4.2\times10^{-3})$ \\
Ex.~2 & $10^{-4}$ & $100.0\%$
& $3.44\times10^{-4}\;(2.0\times10^{-4})$
& $1.14\times10^{-2}\;(5.5\times10^{-3})$
& $1.14\times10^{-2}\;(5.5\times10^{-3})$
& $5.94\times10^{-2}\;(8.5\times10^{-4})$ \\
Ex.~2 & $10^{-5}$ & $100.0\%$
& $8.83\times10^{-5}\;(3.8\times10^{-5})$
& $1.15\times10^{-3}\;(2.0\times10^{-3})$
& $1.15\times10^{-3}\;(2.0\times10^{-3})$
& $3.83\times10^{-2}\;(4.9\times10^{-4})$ \\
\midrule
Ex.~3 & $10^{-4}$ & $100.0\%$
& $2.03\times10^{-3}\;(1.0\times10^{-3})$
& $1.96\times10^{-2}\;(1.7\times10^{-2})$
& $1.96\times10^{-2}\;(1.7\times10^{-2})$
& $6.36\times10^{-2}\;(5.6\times10^{-4})$ \\
Ex.~3 & $10^{-5}$ & $100.0\%$
& $3.10\times10^{-4}\;(2.3\times10^{-4})$
& $2.94\times10^{-3}\;(2.5\times10^{-3})$
& $2.94\times10^{-3}\;(2.5\times10^{-3})$
& $4.11\times10^{-2}\;(4.6\times10^{-4})$ \\
Ex.~3 & $10^{-6}$ & $100.0\%$
& $3.15\times10^{-4}\;(2.7\times10^{-5})$
& $6.30\times10^{-4}\;(2.0\times10^{-4})$
& $6.30\times10^{-4}\;(2.0\times10^{-4})$
& $2.68\times10^{-2}\;(2.9\times10^{-4})$ \\
\bottomrule
\end{tabular}

\medskip
\begin{minipage}{0.97\linewidth}
\scriptsize
$E_{\rm bp}$ and $E_{\rm LFE}^{\rm succ}$ are conditional on successful
recovery of the correct breakpoint count.  $E_{\rm LFE}^{\rm all}$ includes
all trials and therefore also captures occasional failures of structural
identification.  The Gaussian standard deviation is $\delta/\sqrt{3}$,
matching the RMS of the bounded uniform perturbation.
\end{minipage}
\end{table*}

Overall, TV is an effective high-noise baseline, whereas the localized
method is substantially more accurate once the interfaces are reliable. The
low-noise decrease is consistent with
Theorem~\ref{thm:exact_breakpoint_convergence}, and the global/smooth-region
gap agrees with Remark~\ref{rem:breakpoint_location_error}. No
source-dependent parameter adjustment is used.

\section{Conclusions}\label{sec:conclusion}

The inverse source map was decomposed into regularized second-order
differentiation and a boundedly invertible correction. Because the correction
between $-g''$ and $f$ is analytic, the two quantities have identical
interior singular support and jump data. A localized Fourier extension
realization combines staggered breakpoint detection, GTSVD-regularized
coefficients, structure-aligned differentiation, and an exponentially smooth
spectral correction. Exact-partition convergence, finite-dimensional peak
stability, and interface-mismatch estimates describe the main error sources.

Numerically, TV is competitive at strong noise, while the localized method
has a clear advantage in the moderate- and low-noise regimes, where it can
use the higher regularity of each smooth branch. Repeated Gaussian tests also
identify the current high-noise stability limit. Future work will study
sharper detector estimates and differential or pseudodifferential
preconditioners for more general and multidimensional inverse source maps.

\section*{Funding}
This work was partly supported by the Natural Science Foundation of Shandong
Province under Grant Nos. ZR2026MS0015 and ZR2025MS28.
\section*{Data availability}
No data was used for the research described in the article.

\end{document}